\documentclass[12pt]{article}

\usepackage{enumitem}

\usepackage{amssymb, amsmath, amsthm}
\usepackage[all]{xy}

\usepackage{mathptmx}
\usepackage[scaled=.90]{helvet}
\usepackage{courier}

\usepackage{datetime}

\usepackage[pdftex]{hyperref}

\hypersetup{
  colorlinks = true,
  linkcolor = black,
  urlcolor =  blue,
  citecolor = black,
}

\renewcommand{\phi}{\varphi}

\newcommand{\bbZ}{\mathbb{Z}}

\newcommand{\rmA}{\mathrm{A}}

\newcommand{\rmC}{\mathrm{C}}
\newcommand{\rmE}{\mathrm{E}}
\newcommand{\rmF}{\mathrm{F}}

\newcommand{\rmI}{\mathrm{I}}

\newcommand{\rmQ}{\mathrm{Q}}
\newcommand{\rmS}{\mathrm{S}}

\newcommand{\bfC}{\mathbf{C}}

\newcommand{\bfG}{\mathbf{G}}
\newcommand{\bfI}{\mathbf{I}}
\newcommand{\bfK}{\mathbf{K}}
\newcommand{\bfQ}{\mathbf{Q}}
\newcommand{\bfR}{\mathbf{R}}
\newcommand{\bfS}{\mathbf{S}}
\newcommand{\bfT}{\mathbf{T}}

\newcommand{\ab}{\text{\upshape ab}}

\DeclareMathOperator{\pr}{pr}
\DeclareMathOperator{\id}{id}

\DeclareMathOperator{\GL}{GL}
\DeclareMathOperator{\Ab}{Ab}
\DeclareMathOperator{\Aut}{Aut}

\DeclareMathOperator{\Der}{Der}

\DeclareMathOperator{\Hom}{Hom}

\newtheorem{theorem}{Theorem}
\newtheorem*{thm-}{Theorem}
\newtheorem{proposition}[theorem]{Proposition}

\theoremstyle{definition}
\newtheorem*{def-}{Definition}
\newtheorem{definition}[theorem]{Definition}
\newtheorem{example}[theorem]{Example}
\newtheorem{examples}[theorem]{Examples}
\newtheorem{remark}[theorem]{Remark}

\numberwithin{theorem}{section}
\numberwithin{equation}{section}
\numberwithin{figure}{section}

\title{Alexander--Beck modules detect the unknot}

\author{Markus Szymik}

\newdateformat{mydate}{\monthname~\twodigit{\THEYEAR}}
\date{\mydate\today}

\begin{document}

\maketitle

\renewcommand{\abstractname}{}

\vspace{-2\baselineskip}

\begin{abstract}
\noindent 
We introduce the Alexander--Beck module of a knot as a canonical refinement of the classical Alexander module, and we prove that this new invariant is an unknot-detector.

\vspace{\baselineskip}

\vspace{\baselineskip}
\end{abstract}


\section*{Introduction}

One of the most basic and fundamental invariants of a knot~$K$ inside the~$3$--sphere is the knot group~$\pi K$, the fundamental group of the knot complement. Any regular projection~(i.e.,~any~`diagram') of the knot gives rise to a presentation of the group~$\pi K$ in terms of generators and relations, the Wirtinger presentation. As a consequence of Papakyriakopoulos' work~\cite{Papakyriakopoulos}, the knot group~$\pi K$ detects the unknot, but many pairs of distinct knots have isomorphic knot groups. In addition to this, groups are also somewhat complicated objects to manipulate, because of their non-linear nature. There are, therefore, more than enough reasons to look for other knot invariants. 

It turns out that the knot complement is a classifying space of the knot group, see~\cite{Papakyriakopoulos} again. As a consequence, the homology of the group is isomorphic to the homology of the complement. By duality, this homology~(and even the stable homotopy type) is easy to compute, and independent of the knot. Therefore, homology and other abelian invariants of groups and spaces do not give rise to exciting knot invariants unless we also find a way to refine the strategy to some extent. 

For instance, we can extract the Alexander polynomial of a knot from the homology of the canonical infinite cyclic cover of the knot complement. However, many knots have the same Alexander polynomial as the unknot. Such examples include
the Seifert knot~\cite{Seifert}, 
all untwisted Whitehead doubles~\cite{Whitehead} of knots,
the Kinoshita--Terasaka knot~\cite{Kinoshita+Terasaka}, and
the Conway knot~\cite{Conway}. See also~\cite{Garoufalidis+Teichner}.~(At the time of writing, the corresponding problem for the Jones polynomial appears to be open.) This state of affairs may suggest that groups and the invariants derived from their abelianizations are not the most efficient algebraic means to provide invariants of knots. 

Building on fundamental work~\cite{Waldhausen} of Waldhausen, Joyce~\cite{Joyce} and Mat\-veev~\cite{Matveev} have independently shown that there is an algebraic structure that gives rise to a complete invariant of knots~$K$: the knot quandle~$\rmQ K$. As with the knot group, we can describe it in terms of paths in the knot complement, and we can present it using any of the knot's diagrams. 

The knot quandle functorially determines the knot group, and the classical Alexander module of a knot~$K$ has a comeback in the~(absolute) abelianization of the knot quandle~$\rmQ K$. 

In this paper, we use a relative version of the abelianization functor that goes back to Beck~\cite{Beck} to introduce a refinement of the classical Alexander module. The following will reappear as Definition~\ref{def:AB_K}, after an explanation of the terms involved.

\begin{def-}
Let~$K$ be a knot with knot quandle~$\rmQ K$. The {\em Alexander--Beck module} of~$K$ is the value  at the terminal object of the left adjoint of the forgetful functor from~$\rmQ K$--modules to quandles over~$\rmQ K$.
\end{def-}

As the name suggests, the Alexander--Beck module of a knot is a linear algebraic object and therefore embedded into a more accessible context than the knot group. We can also compute it from any diagram of the knot.

While knots are classified in theory, as recalled above, through their associated quandles, there is still considerable interest in finding weaker invariants. Of course, we do not like the invariants to be too weak, like homology or the Alexander polynomial. They should at least be unknot-detectors. As we have noted above, the knot group is such an invariant, and in~\cite{KM} this property is established for Khovanov homology, a refinement of the Jones polynomial. We will see that there is no need to introduce homology, because of the main result of this paper, Theorem~\ref{thm:main}:

\begin{thm-}
A knot is trivial if and only if its Alexander--Beck module is free.
\end{thm-}

The outline of this paper is as follows. In the following Section~\ref{sec:abel}, we will review the categories of abelian group objects for algebraic theories. In Section~\ref{sec:abel_quan}, we apply these ideas to the theories of racks and quandles. In Section~\ref{sec:abel_knot}, we specialize this further to the knot quandles and explain how we can interpret the classical Alexander invariants from this point of view. Section~\ref{sec:Beck} contains a review of Beck modules over objects for any algebraic theory. In Section~\ref{sec:beck_quan}, we employ this for the theories of racks and quandles. In Section~\ref{sec:beck_knot}, we specialize again to knot quandles. We introduce the Alexander--Beck modules in Definition~\ref{def:AB_K}, and we prove the main result of this paper, Theorem~\ref{thm:main}. 

Homology and the derived functors of the abelianization are addressed elsewhere~\cite{Szymik:Q=Q}.


\section{Abelian group objects}\label{sec:abel}

Racks and quandles can be studied in the context of (one-sorted) algebraic theories in the sense of Lawvere~\cite{Lawvere}. This was done, for instance, in~\cite{Szymik}. Other more standard examples of such theories are given by the theory of groups, the theory of rings, the theory of sets with an action of a given group~$G$, the theory of modules over a given ring~$A$, the theory of Lie algebras, and not to forget the initial theory of sets.  In this section we review the categories of abelian group objects in algebraic theories.

An algebraic theory is usually given by a small category that codifies the operations and their relations. Then the {\it models} (or {\it algebras}) of that algebraic theory are certain functors from that small category to the large category of sets, see~\cite{Lawvere}. For example, a group $G$ in the traditional sense defines a model for the algebraic theory of groups by means of a functor that sends the free group on $n$ generators to the set $G^n$.

For any algebraic theory, the category~$\bfT$ of its models is large, and all categories in this paper will be of that size. They are complete, cocomplete, and have a `small' and~`projective' generator: a free model on one generator. The class of~`effective epimorphisms' agrees with the class of surjective homomorphisms. We will write~$\bfS$,~$\bfG$,~$\bfR$, and~$\bfQ$ for the category of sets, groups, racks~(see Definition~\ref{def:rack}), and quandles~(see Definition~\ref{def:quandle}), respectively. Whenever we pick a category~$\bfT$ of models for an algebraic theory, the reader is invited to choose any of these for guidance.

There are forgetful functors among these categories that all have left adjoints.
\[
\bfS\longrightarrow
\bfR\longrightarrow
\bfQ\longrightarrow
\bfG
\]
In particular, the left adjoint~$\bfS\to\bfT$ to the forgetful functor sends a set~$S$ to the free model~$\rmF\bfT(S)$ on the set~$S$ of generators. We will write~$\rmF\bfT_n$ if~$S$ is the set~$\{1,\dots,n\}$ with~$n$ elements. 
 
\begin{definition}\label{def:ag}
If~$\bfC$ is a category with finite products, an {\em abelian group object} in~$\bfC$ is an object~$M$ together with operations~$e\colon\star\to M$ (the unit), \hbox{$i\colon M\to M$} (the inverse), and~\hbox{$a\colon M\times M\to M$} (the addition) such that, writing~$e'$ for the composition of~$e$ with the unique map~$M\to\star$, the four diagrams
\[
\xymatrix{
M\times M\times M\ar[r]^-{\id\times a}\ar[d]_{a\times\id}&M\times M\ar[d]^a\\
M\times M\ar[r]_a&M
}
\hspace{1em}
\xymatrix{
M\times M\ar[dr]_a\ar[rr]^{(\pr_2,\pr_1)}&&M\times M\ar[dl]^a\\
&M&
}
\]
\[
\xymatrix{
M\ar[r]^-{(\id,e')}\ar[dr]_\id&M\times M\ar[d]^a&M\ar[l]_-{(e',\id)}\ar[dl]^\id\\
&M&
}
\hspace{1em}
\xymatrix{
M\ar[r]^-{(\id,i)}\ar[dr]_{e'}&M\times M\ar[d]_a&M\ar[l]_-{(i,\id)}\ar[dl]^{e'}\\
&M&
}
\]
commute.
\end{definition}

See Beck's thesis~\cite{Beck} and Quillen's summary in~\cite{Quillen:summary}. 

\begin{remark}\label{rem:abelian}
A useful way of rephrasing the Definition~\ref{def:ag} goes as follows: an abelian group structure on an object~$M$ is a lift of the set-valued pre\-sheaf~\hbox{$C\mapsto\bfC(C,M)$} on~$\bfC$ that is represented by~$M$ to an abelian presheaf, that is to a presheaf that takes values in the category of abelian groups.
\end{remark}

Let~$\bfT$ be the category of models (or algebras) for an algebraic theory. The category~$\Ab(\bfT)$ of abelian group objects in~$\bfT$ is equivalent to the category of models for the tensor product of the given theory with the theory of abelian groups~\cite{Freyd}. The category~$\Ab(\bfT)$ is also equivalent to the category of modules over a ring, the endomorphism ring of its generator. We will denote this ring by~$\bbZ\bfT$.

The category~$\Ab(\bfT)$ of abelian group objects in~$\bfT$ comes with a faithful forgetful functor~$\Ab(\bfT)\to\bfT$, and that functor has a left adjoint
\[
\Omega\colon\bfT\longrightarrow\Ab(\bfT).
\]
This will be referred to as the {\em (absolute) abelianization} functor. 

\begin{example}
If~$\bfT=\bfG$ is the category of groups, then~$\Ab(\bfG)$ is the full subcategory of abelian groups, or~$\bbZ$--modules, so that~$\bbZ\bfG=\bbZ$. The abelianization of a group~$G$ in the abstract sense discussed above is just the abelianization~$G^\ab$ of that group in the sense of group theory: the quotient of~$G$ by its commutator subgroup.
\end{example}

There is a standard recipe to compute the abelianization~$\Omega(X)$, at least in principle. 

\vbox{\begin{proposition}
If the diagram
\[
\xymatrix@1{
X&\rmF\bfT(S)\ar[l]&\rmF\bfT(R)\ar@<+.4ex>[l]\ar@<-.4ex>[l]
}
\]
displays~$X$ as a coequalizer of free objects~$\rmF\bfT(S)$ and~$\rmF\bfT(R)$, then 
there is a diagram
\[
\xymatrix@1{
\Omega(X)&\bbZ\bfT(S)\ar[l]&\bbZ\bfT(R)\ar[l]
}
\]
that displays~$\Omega(X)$ as a cokernel of the difference of the induced maps between free~$\bbZ\bfT$--modules.
\end{proposition}}

In order to find such a coequalizer diagram, choose a presentation of~$X$ by generators~$S$ and relations~$R$, or simply take the canonical one, with~$S=X$ and~$R=\rmF\bfT X$.

\begin{proof}
If the model~$X=\rmF\bfT(S)$ happens to be free on a set~$S$, then we would like an abelian model~$\Omega(\rmF\bfT(S))$ together with natural isomorphisms
\[
\Hom_{\Ab(\bfT)}(\Omega(\rmF\bfT(S)),M)\cong\bfT(\rmF\bfT(S),M)\cong\bfS(S,M)
\]
for all abelian models~$M$. There exists a free abelian model on any given set~$S$, because such a model corresponds to a free~$\bbZ\bfT$--module~$\bbZ\bfT(S)$. It is then clear that such a free abelian model solves our problem in this special case, that is when~$X$ is free. 

In general, we assume that the object~$X$ comes presented as a colimit of free objects, and take into account that the functor~$\Omega$, as a left adjoint, has to preserve these. Specifically, we assume that we have a diagram
\[
\xymatrix@1{
X&\rmF\bfT(S)\ar[l]&\rmF\bfT(R)\ar@<+.4ex>[l]\ar@<-.4ex>[l]
}
\]
that displays~$X$ as a coequalizer of free objects~$\rmF\bfT(S)$ and~$\rmF\bfT(R)$.  Then we get a coequalizer diagram
\[
\xymatrix@1{
\Omega(X)&\Omega(\rmF\bfT(S))\ar[l]&\Omega(\rmF\bfT(R))\ar@<+.4ex>[l]\ar@<-.4ex>[l]
}
\]
of abelian models, so that the abelianization~$\Omega(X)$ is the cokernel of the difference of the parallel maps. This presents~$\Omega(X)$ in terms of free abelian models, as desired.
\end{proof}

\begin{remark}
On the level of automorphism groups, abelianization induces homomorphisms
\begin{equation}\label{eq:Burau}
\Aut_{\bfT}(\rmF\bfT_n)\longrightarrow\GL_n(\bbZ\bfT)
\end{equation}
from the automorphism groups of the free models into the general linear groups over the ring~$\bbZ\bfT$. As we will see below in Remark~\ref{rem:Burau}, these can be thought of as generalizations of the Burau representations.
\end{remark}

\begin{remark}
Let~$\bfT$ be a category of models for an algebraic theory, let~$X$ be an object of~$\bfT$, and let~$M$ in~$\Ab(\bfT)$ be an abelian model. The sets
\[
\bfT(X,M)\cong\Hom_{\Ab(\bfT)}(\Omega(X),M)
\]
are actually abelian groups, see Remark~\ref{rem:abelian}. We will write~$\Der(X;M)$ for either of them. The elements are the~{\em derivations} in the sense of Beck. See again~\cite{Beck} and~\cite{Quillen:summary}. There is a universal derivation~$X\to\Omega(X)$, adjoint to the identity.
\end{remark}


\section{Abelian racks and quandles}\label{sec:abel_quan}

In this section we show how the general concepts from the previous section apply to the theory of racks and quandles. Although there is no claim to originality here, there is nevertheless reason to require a reasonably self-contained exposition: we can use it to fix the notation used throughout the text; it is instructive to see the general concepts of the previous section worked out in the case of interest to us; and the exact statements given here cannot be conveniently referenced from the literature.

\begin{definition}\label{def:rack}
A {\em rack}~$(R,\rhd)$ is a set~$R$ together with a binary operation~$\rhd$ such that all left multiplications
\[
R\longrightarrow R,\,y\longmapsto x\rhd y
\]
are automorphisms, i.e.~they are bijective and satisfy
\[
x\rhd(y\rhd z)=(x\rhd y)\rhd(x\rhd z)
\]
for all~$x$,~$y$, and~$z$.
\end{definition}

The invertibility condition can be encoded via another binary operation. See the papers by Brieskorn~\cite{Brieskorn} and Fenn--Rourke~\cite{Fenn+Rourke}.

\begin{definition}\label{def:abelian}
An {\em abelian rack} is an abelian group object in the category~$\bfR$ of racks, or equivalently, a rack object in abelian groups. 
\end{definition}

\begin{remark}\label{rem:abelian_in_Joyce}
This meaning of `abelian' is different from the one in~\cite[Def.~1.3]{Joyce}, where Joyce considers the equation
\[
(w\rhd x)\rhd(y\rhd z)=(w\rhd y)\rhd(x\rhd z).
\]
This equation characterizes `rack objects' in the category~$\bfR$ of racks, rather than abelian group objects: a {\it rack object} in the category of racks is a rack~$R$ such that the binary operation~\hbox{$\rhd\colon R\times R\to R$} is a morphism of racks.
\end{remark}

An abelian rack (as in Definition~\ref{def:abelian}) is a rack~$M$ that is also an abelian group~(with zero~$0$), and both structures are compatible in the sense that the map~$M\times M\to M$ that sends~$(x,y)$ to~$x\rhd y$ is a group homomorphism~(with respect to the usual abelian group structure on the product). In equations, this means~$0\rhd 0=0$ and
\[
(m+n)\rhd(p+q)=(m\rhd p)+(n\rhd q)
\]
for all~$m, n, p, q$ in~$M$. In particular, we have an automorphism~$\alpha\colon M\to M$ of the abelian group~$M$ defined by
\[
\alpha(x)=0\rhd x,
\]
and an endomorphism~$\epsilon\colon M\to M$ defined by
\[
\epsilon(x)=x\rhd 0.
\] 
The equation~\hbox{$x\rhd y=x\rhd0+0\rhd y$} can then be rewritten
\begin{equation}\label{eq:sum}
x\rhd y=\epsilon(x)+\alpha(y).
\end{equation}
We see that these two morphisms determine the rack structure and conversely. The calculation
\begin{align*}
\epsilon(\alpha(y))
&=\epsilon(0\rhd y)\\
&=(0\rhd y)\rhd 0\\
&=(0\rhd y)\rhd (0\rhd 0)\\
&=0\rhd(y\rhd 0)\\
&=\alpha(\epsilon(y))
\end{align*}
shows that~$\alpha$ and~$\epsilon$ commute. 

\begin{proposition}
The category of abelian racks is equivalent to the category of modules over the ring
\[
\bbZ\bfR=\bbZ[\rmA^{\pm},\rmE]/(\rmE^2-\rmE(1-\rmA)).
\]
\end{proposition}

\begin{proof}
Fenn and Rourke~\cite[Sec.~1, Ex.~6]{Fenn+Rourke} have remarked that~$\bbZ\bfR$--modules define racks, using~\eqref{eq:sum}. Conversely, the calculation
\begin{align*}
\epsilon(x)
&=x\rhd0\\
&=x\rhd(0\rhd 0)\\
&=(x\rhd0)\rhd(x\rhd 0)\\
&=\epsilon(x)\rhd \epsilon(x)\\
&=(\epsilon(x)\rhd0)+(0\rhd \epsilon(x))\\
&=\epsilon^2(x)+\alpha\epsilon(x)
\end{align*}
shows that any abelian rack admits a module structure over that ring.
\end{proof}

\begin{remark}
If~$X$ is a rack, its (absolute) abelianization~$\Omega(X)$ corresponds to the quotient of the free~$\bbZ\bfR$--module with basis~$X$ by the relations
\[
x\rhd y=\rmE x+\rmA y
\]
for~$x$ and~$y$ in~$X$. For instance, if~$\star$ is the terminal rack, then it has precisely one element, and we get that~$\Omega(\star)$ is the quotient of the ring~$\bbZ\bfR$ by the ideal generated by the element~\hbox{$1=\rmE+\rmA$}. This is the ring~$\bbZ[\rmA^{\pm}]$, with~$\rmE=1-\rmA$.
\end{remark}

Every rack~$R$ comes with a canonical automorphism~$\rmF_R$ that is defined by the simple equation~\hbox{$\rmF_R(x)=x\rhd x$}, see~\cite{Szymik} for an exhaustive study. 

\begin{definition}\label{def:quandle}
A {\em quandle} is a rack such that its canonical automorphism is the identity. 
\end{definition}

The theory of quandles was born in the papers of Joyce~\cite{Joyce} and Matveev \cite{Matveev}, after a considerable embryonal phase. For details, we refer to the original papers and~\cite{Fenn+Rourke} again.

\begin{proposition}\label{prop:ab_quan}
The category of abelian quandles is equivalent to the category of modules over the ring
\[
\bbZ\bfQ=\bbZ[\rmA^\pm].
\]
\end{proposition}

\begin{proof}
Again, it is well known that~$\bbZ\bfQ$--modules define quandles. See~\cite[Sec.~1, p.~38]{Joyce}, \cite[\S~2, Ex.~1]{Matveev}, or~\cite[Sec.~1, Ex.~5]{Fenn+Rourke}, for instance. Conversely, the quandle condition~\hbox{$x\rhd x=x$} implies
\[
x=x\rhd x= x\rhd0+0\rhd x=\epsilon(x)+\alpha(x),
\]
or
\[
\epsilon=\id-\alpha.
\]
This leads to the relation~$\epsilon^2=\epsilon(1-\alpha)$ for abelian quandles.
\end{proof}

\begin{definition}
A rack is {\em involutary} if the axiom~$x\rhd(x\rhd y)$ is satisfied. A {\em kei} is an involutary quandle.
\end{definition}

If~$M$ is an abelian involutary rack, then we have~$\alpha^2=\id$. This implies the following two results.

\begin{proposition}
The category of abelian involutary racks is equivalent to the category of modules over the ring
\[
\bbZ\bfI=\bbZ[\rmA^{\pm},\rmE]/(\rmE^2-\rmE(1-\rmA),\rmA^2-1).
\] 
\end{proposition}

\begin{proposition}
The category of abelian kei is equivalent to the category of modules over the ring
\[
\bbZ\bfK=\bbZ[\rmA]/(\rmA^2-1).
\]
\end{proposition}

\begin{remark}
Note that in the two idempotent cases (when the canonical automorphism is trivial) we get group rings~\hbox{$\bbZ\bfQ\cong\bbZ[\,\rmC_\infty\,]$} and~\hbox{$\bbZ\bfK\cong\bbZ[\,\rmC_2\,]$}, where~$\rmC_n$ denotes a cyclic group of order~$n$.
\end{remark}

\begin{remark}\label{rem:Burau}
The Burau representations~\eqref{eq:Burau} for the theories~$\bfR$,~$\bfQ$,~$\bfI$, and~$\bfK$ take the following form.
\begin{align*}
\Aut_\bfR(\rmF\bfR_n)&\longrightarrow\GL_n(\bbZ[\rmA^{\pm},\rmE]/(\rmE^2-\rmE(1-\rmA)))\\
\Aut_\bfQ(\rmF\bfQ_n)&\longrightarrow\GL_n(\bbZ[\rmA^{\pm1}])\\
\Aut_\bfI(\rmF\bfI_n)&\longrightarrow\GL_n(\bbZ[\rmA^{\pm},\rmE]/(\rmE^2-\rmE(1-\rmA),\rmA^2-1))\\
\Aut_\bfK(\rmF\bfK_n)&\longrightarrow\GL_n(\bbZ[\rmA^{\pm}]/(\rmA^2-1))
\end{align*}
The name comes from the fact that the braid group on~$n$ strands embeds into the group~$\Aut_\bfQ(\rmF\bfQ_n)$ in such a way that the restriction of the representation above is the classical Burau representation~\cite{Burau}.
\end{remark}


\section{The classical Alexander modules of knots}\label{sec:abel_knot}

In this section we apply the theory of abelian quandles to the quandles arising in knot theory. We will also see how the classical Alexander invariants can be interpreted in these terms.

Let~$K$ be a knot in the~$3$--sphere~$\rmS^3$ with knot quandle~$\rmQ K$. The knot quandle has been introduced by Joyce~\cite{Joyce} and Matveev~\cite{Matveev}, who have also shown, building on deep results of Waldhausen's~\cite{Waldhausen}, that the knot quandle is a complete invariant of the knot.

We can find a presentation of the knot quandle~$\rmQ K$ from any diagram of the knot: the generators are the arcs, and there is a relation of the form~\hbox{$x\rhd y = z$} whenever~$x$,~$y$, and~$z$ meet in a crossing, with~$x$ as the overpass, and~$y$ turns into~$z$ under it, as in Figure~\ref{fig:crossing}. Note that the orientations of the arcs~$y$ and~$z$ are not relevant and they have, for that reason, not been displayed.

\begin{figure}
\caption{A crossing in a knot diagram}
\label{fig:crossing}
\[
\xymatrix{
x && y\\
&{\phantom{x}}\ar@{-}[ur]&\\
z\ar@{-}[ur] && x\ar[lluu]
}
\]
\end{figure}
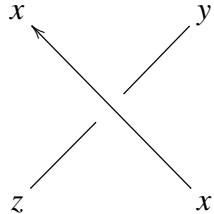

In other words, we can present the quandle~$\rmQ K$ as a coequalizer
\begin{equation}\label{eq:presQ}
\xymatrix@1{
\rmQ K&\rmF\bfQ\{\text{arcs}\}\ar[l]&\rmF\bfQ\{\text{crossings}\},\ar@<+.4ex>[l]\ar@<-.4ex>[l]
}
\end{equation}
where one of the two parallel arrows sends a crossing involving the arcs~$x$,~$y$, and~$z$ as in Figure~\ref{fig:crossing} to the element~$x\rhd y$ in the free quandle on the set of arcs, and the other one sends it to the arc~$z$.

There is a forgetful functor from the category of groups to the category of quandles: given a group~$G$, the underlying set comes with the quandle structure given by~\hbox{$x\rhd y=xyx^{-1}$}. This functor admits a left adjoint. The left adjoint sends a quandle~$Q$ to the quotient of the free group on~$Q$ with the relations~$x\rhd y=xyx^{-1}$ for all~$x$ and~$y$ in~$Q$, see~\cite[Sec.~6]{Joyce} or~\cite[\S~5]{Matveev}. The left adjoint sends a knot quandle~$\rmQ K$ to the knot group~$\pi K$, because the presentation~\eqref{eq:presQ} above is mapped to Wirtinger's~(unpublished) presentation of the knot group, where the relation~$x\rhd y=z$ now reads~$xyx^{-1}=z$, see~\cite[Sec.~6]{Joyce} or~\cite[Prop.~3]{Matveev}.

The abelianization~$\pi^\ab K$ of the knot group is always infinite cyclic, independent of the knot~$K$. There are different knots with isomorphic knot groups, but at least there is no non-trivial knot with a knot group isomorphic to that of the unknot: by~\cite{Papakyriakopoulos}, the knot group is abelian~(hence infinite cyclic) if and only if the knot is trivial. 

The weakness of the knot group as an invariant has its virtues: the kernel of the abelianization~$\pi K\to\pi^\ab K$ defines an infinite cyclic covering of the knot complement, and the first homology of the covering space, as a module over the group ring~$\bbZ[\rmA^{\pm1}]$ of the deck transformation group, is the classical Alexander module of the knot. The Alexander polynomial is the characteristic polynomial of the action of the generator~$\rmA$ on the torsion part. See Milnor's concise summary~\cite{Milnor} or Lickorish's~\cite[Chp.~6]{Lickorish} for a textbook treatment.

The abelianization~$\Omega(\rmQ K)$ of the knot quandle~$\rmQ K$ is an abelian quandle, or equivalently, according to Proposition~\ref{prop:ab_quan}, it can also be thought of as a module over the Laurent polynomial ring~$\bbZ\bfQ=\bbZ[\rmA^{\pm1}]$. The general formalism of Section~\ref{sec:abel} allows us to compute the module~$\Omega(\rmQ K)$ from a presentation given by a diagram of the knot as follows.

\vbox{\begin{proposition}
Given a diagram of a knot~$K$, the~$\bbZ[\rmA^{\pm1}]$--module~$\Omega(\rmQ K)$ is a cokernel of the homomorphism
\[
\bbZ[\rmA^{\pm1}]\{\text{\upshape arcs}\}
\longleftarrow
\bbZ[\rmA^{\pm1}]\{\text{\upshape crossings}\}
\]
between free~$\bbZ[\rmA^{\pm1}]$--modules that sends a crossing to the~$\bbZ[\rmA^{\pm1}]$--linear combination
\[
(1-\rmA)x+\rmA y-z
\]
of the arcs involved in that crossing as in Figure~\ref{fig:crossing}.
\end{proposition}}

\begin{proof}
Recall that the knot quandle~$\rmQ K$ has a presentation as a coequalizer~\eqref{eq:presQ}. Abelianization~$\Omega$ is a left adjoint, and it therefore preserves colimits such as coequalizers. This immediately leads to the stated result.
\end{proof}

Recall that the Alexander module of a knot is the~$\bbZ[\rmA^{\pm1}]$--module given by the first homology of the infinite cyclic cover of the knot. The {\em extended Alexander module} will have an additional free~$\bbZ[\rmA^{\pm1}]$--summand of rank~$1$. In other words, if~$M$ is the Alexander module, then~$\bbZ[\rmA^{\pm1}]\oplus M$ is the extended Alexander module. The following result has been established, in a different mathematical language, by Joyce~\cite[Sec.~17]{Joyce} and Matveev~\cite[\S11]{Matveev}, compare Remark~\ref{rem:abelian_in_Joyce}.

\begin{proposition}\label{prop:classical} 
The abelianization of~$\rmQ K$, the~$\bbZ[\rmA^{\pm1}]$--module~$\Omega(\rmQ K)$, is isomorphic to the extended Alexander module of the knot. 
\end{proposition}

\begin{example}
The unknot~$U$ has a diagram with one arc and no crossing. Therefore, its quandle~$\rmQ U$ is the free quandle on one generator. The free quandle on one generator is the terminal quandle with one element. Its abelianization is a free~$\bbZ[\rmA^{\pm1}]$--module on one generator. It can also be described as the cokernel of the homomorphism~\hbox{$\bbZ[\rmA^{\pm1}]\longleftarrow0$}.
\end{example}

\begin{example}
For the trefoil knot~$T$, the usual diagram with three arcs and three crossings leads to the presentation matrix
\[
\begin{bmatrix}
1-\rmA & -1 & \rmA\\
\rmA & 1-\rmA & -1 \\
-1 & \rmA & 1-\rmA 
\end{bmatrix}
\]
for the abelianization~$\Omega(\rmQ T)$, the extended Alexander module of~$T$, as a cokernel of a~$\bbZ[\rmA^{\pm1}]$--linear endomorphism of~$\bbZ[\rmA^{\pm1}]^{\oplus 3}$. We recognize that the extended Alexander module is isomorphic to
\[
\bbZ[\rmA^{\pm1}]/(\rmA^2-\rmA+1)\oplus\bbZ[\rmA^{\pm1}].
\] 
\end{example}

\begin{remark}
When the Alexander module of a knot~$K$ is isomorphic to the cyclic module~$\bbZ[\rmA^{\pm1}]/(\Delta_K(\rmA))$, where~$\Delta_K$ is the Alexander polynomial of~$K$, we have
\[
\Omega(\rmQ K)\cong\bbZ[\rmA^{\pm1}]/(\Delta_K(\rmA))\oplus\bbZ[\rmA^{\pm1}].
\]
But not every Alexander module is cyclic. For instance, the one of the pretzel knot~$P(25,-3,13)$ is not, see~\cite{Fox+Smythe} and~\cite{Milnor}.
\end{remark}

\begin{remark}
The presence of the extraneous free~$\bbZ[\rmA^{\pm1}]$--summand can be explained as follows: if a quandle~$Q$ is non-empty, then it has the free quandle of rank~$1$ as a retract. (Since~$\rmF\bfQ_1$ is a singleton~$\star$, any element in~$Q$ gives a section of the projection~\hbox{$Q\to\star$}.) As a consequence of that, the abelianization~$\Omega(Q)$ will then have the free module~$\Omega(\rmF\bfQ_1)$ of rank~$1$ as a retract.
\end{remark}


\section{Beck modules}\label{sec:Beck}

Many categories~$\bfT$ of models for algebraic theories do not have interesting abelian group objects. For instance, this is the case for the theory of commutative rings. And even for those theories that have interesting abelian group objects, one can do better than just looking at the absolute abelianization functor: one can use a relative version of it. This leads to abelian invariants that are tailored to a given object~$X$ of~$\bfT$. For instance, given a group~$G$, this naturally leads us to consider~$G$--modules, that is, modules over the integral group ring~$\bbZ G$. In general, the situation is more complicated though.

Let us choose an object~$X$ in~$\bfT$, and let~$\bfT_X$ denote the slice category of objects of~$\bfT$ over~$X$. The product of $\mu\colon M\to X$ and $\nu\colon N\to X$ in~$\bfT_X$ is the pullback
\[
M\times_XN=\{\,(m,n)\in M\times N\,|\,\mu(m)=\nu(n)\,\}
\]
in~$\bfT$. Let again~$\Ab(\bfT_X)$ denote the category of abelian group objects in the category~$\bfT_X$. According to Quillen~\cite[p.~69]{Quillen:summary}, this category is abelian.

\begin{definition}
The objects in~$\Ab(\bfT_X)$ are the~{\em$X$--modules} in the sense of Beck~\cite[Def.~5]{Beck}. 
\end{definition}

In other words, an $X$--module consists of the following data: a $\bfT$--model $M$ together with a structure morphism~\hbox{$M\to X$} and operations~\hbox{$e\colon X\to M$} (the unit),~$i\colon M\to M$ (the inverse), and~\hbox{$a\colon M\times_X M\to M$} (the addition). These data need to satisfy the following axioms: writing~$e'$ for the composition of~$e$ with the structure morphism~$M\to X$, the four diagrams
\[
\xymatrix{
M\times_X M\times_X M\ar[r]^-{\id\times a}\ar[d]_{a\times\id}&M\times_X M\ar[d]^a\\
M\times_X M\ar[r]_a&M
}
\hspace{1em}
\xymatrix{
M\times_X M\ar[dr]_a\ar[rr]^{(\pr_2,\pr_1)}&&M\times_X M\ar[dl]^a\\
&M&
}
\]
\[
\xymatrix{
M\ar[r]^-{(\id,e')}\ar[dr]_\id&M\times_X M\ar[d]^a&M\ar[l]_-{(e',\id)}\ar[dl]^\id\\
&M&
}
\hspace{1em}
\xymatrix{
M\ar[r]^-{(\id,i)}\ar[dr]_{e'}&M\times_X M\ar[d]_a&M\ar[l]_-{(i,\id)}\ar[dl]^{e'}\\
&M&
}
\]
commute.

\begin{example}\label{ex:Gmod}
If~$\bfT=\bfG$ is the category of groups, and~$G$ is any group then the category~$\Ab(\bfG_G)$ is equivalent to the category of~$\bbZ G$--modules. An equivalence is given by associating with every abelian group over~$G$ the kernel of its structure homomorphism to~$G$. An inverse is constructed by sending a~$\bbZ G$--module~$M$ to the split extension~$M\rtimes G$ given by the semi-direct product.
\end{example}

\begin{remark}
Contrary to what one might wish for, the abelian category~$\Ab(\bfT_X)$ is not always equivalent to the category of modules over a ring. For instance, when~$\bfT=\bfS$ is the algebraic theory of sets, then~$\bfT_X$ is the category of sets over~$X$, or equivalently the category of~$X$--graded sets, by passage to fibers.
Then~$\Ab(\bfT_X)$ is the category of~$X$--graded abelian groups, and the author knows how to realize this~(up to equivalence) as a category of modules over a ring only if~$X$ is finite. This problem disappears if one is willing to work with `ringoids.'
\end{remark}

\begin{definition}
The left adjoint
\[
\Omega_X\colon\bfT_X\longrightarrow\Ab(\bfT_X).
\]
to the forgetful functor~$\Ab(\bfT_X)\to\bfT_X$ is the {\em (relative) abelianization} functor. 
\end{definition}

\begin{remark}
In the relative situation, when~$Y$ in~$\bfT_X$ is an object over~$X$, and~$N$ in~$\Ab(\bfT_X)$ is an~$X$--module in the sense of Beck, we will write $\Der_X(Y;N)$ for the module
\[
\bfT_X(Y,N)\cong\Hom_{\Ab(\bfT_X)}(\Omega_X(Y),N)
\]
of {\em$X$--derivations} from~$Y$ into~$N$. We will mostly be interested in the case when~\hbox{$Y=X$} is the terminal object over~$X$. The reasons for this will become apparent in Remark~\ref{rem:rel_vs_abs} below.
\end{remark}


Let us see how the~$X$--module~$\Omega_X(Y)$ might be computed, at least in principle. 

\begin{proposition}
If the diagram
\[
\xymatrix@1{
Y&\rmF\bfT(S)\ar[l]&\rmF\bfT(R)\ar@<+.4ex>[l]\ar@<-.4ex>[l]
}
\]
displays~$Y$ as a coequalizer of free objects~$\rmF\bfT(S)$ and~$\rmF\bfT(R)$, then there is a diagram
\[
\xymatrix@1{
\Omega_X(Y)&\Omega_X(\rmF\bfT(S))\ar[l]&\Omega_X(\rmF\bfT(R))\ar[l]
}
\]
that displays~$\Omega_X(Y)$ as the cokernel of the difference of the induced maps between free~$X$--modules.
\end{proposition}

\begin{proof}
If the object~$Y=\rmF\bfT(S)$ over~$X$ is free on a set~$S$, then we would like an~$X$--module~$\Omega_{X}(\rmF\bfT(S))$ together with natural isomorphisms
\[
\Hom_{\Ab(\bfT_X)}(\Omega_X(\rmF\bfT(S)),N)\cong\bfT_X(\rmF\bfT(S),N)\cong\bfS_X(S,N)
\]
for all~$X$--modules~$N$. For any given set~\hbox{$S\to X$} over the underlying set of the object~$X$, there exists a free~$X$--module over it. It is again clear that such a free~$X$--module over the composition~\hbox{$S\to\rmF\bfT(S)=Y\to X$} solves our problem. 

In general, we can write the object~$Y$ as a colimit of free objects, and use that the functor~$\Omega_X$, as a left adjoint, has to preserve these. Specifically, if
\[
\xymatrix@1{
Y&\rmF\bfT(S)\ar[l]&\rmF\bfT(R)\ar@<+.4ex>[l]\ar@<-.4ex>[l]
}
\]
displays~$Y$ as a coequalizer of free objects~$\rmF\bfT(S)$ and~$\rmF\bfT(R)$, then we have a coequalizer diagram
\[
\xymatrix@1{
\Omega_X(Y)&\Omega_X(\rmF\bfT(S))\ar[l]&\Omega_X(\rmF\bfT(R))\ar@<+.4ex>[l]\ar@<-.4ex>[l]
}
\]
of~$X$--modules, so that the module~$\Omega_X(Y)$ is the cokernel of the difference of the parallel induced maps. This presents~$\Omega_X(Y)$ in terms of free~$X$--modules, as desired.
\end{proof}


\begin{remark}\label{rem:rel_vs_abs}
One of the most confusing aspects of the theory might be the interplay between absolute and relative abelianizations. Let us explain this a bit. If~$f\colon X\to Y$ is a morphism in~$\bfT$, the pullback defines a functor
\[
f^*\colon\bfT_Y\longrightarrow\bfT_X
\]
that preserves limits. It maps abelian group objects to abelian group objects, so that we also have a functor
\[
\Ab(f^*)\colon\Ab(\bfT_Y)\longrightarrow\Ab(\bfT_X)
\]
that commutes with the forgetful functors. We deserve a diagram.
\[
\xymatrix@C=4em{
\bfT_X&\bfT_Y\ar[l]_-{f^*}\\
\Ab(\bfT_X)\ar[u]&\Ab(\bfT_Y)\ar[l]^{\Ab(f^*)}\ar[u]
}
\]
Both of the functors~$f^*$ and~$\Ab(f^*)$ have left adjoints, say~$f_*$ and~$\Ab(f_*)$. Typically, only the first one is given by composition with~$f$; the second one is only rarely given in this form. In any event, the left adjoints always commute with the abelianization functors as indicated in the following diagram.
\[
\xymatrix@C=4em{
\bfT_X\ar[r]^-{f_*}\ar[d]_{\Omega_X}&\bfT_Y\ar[d]^{\Omega_Y}\\
\Ab(\bfT_X)\ar[r]_{\Ab(f_*)}&\Ab(\bfT_Y)
}
\]
In particular, by evaluation at the identity~$\id_X$, thought of as the terminal object in the category~$\bfT_X$, we get isomorphisms
\begin{equation}\label{eq:rel}
\Omega_Y(X)\cong\Ab(f_*)\Omega_X(X).
\end{equation}
For this reason, it is common to concentrate on the~$X$--module~$\Omega_X(X)$ and refer to the other~$X$--modules~$\Omega_Y(X)$ and~$\Ab(f_*)$ only when needed. For instance, in the extreme case, if~$Y=\star$ is the terminal object in the category~$\bfT$, the relation~\eqref{eq:rel} reads
\begin{equation}\label{eq:determines}
\Omega(X)\cong\Ab(f_*)\Omega_X(X). 
\end{equation}
This correctly suggests that~$\Omega_X(X)$ is the better invariant than~$\Omega(X)$, and our focus will be on it from now on. We can also think of~$f$ as a morphism over~$Y$, and then it induces a homomorphism~$\Omega_Y(X)\to\Omega_Y(Y)$ of~$Y$--modules that can be translated into a homomorphism
\[
\Ab(f_*)\Omega_X(X)\longrightarrow\Omega_Y(Y)
\]
using~\eqref{eq:rel}. These are useful to have around once one commits to working with the~$\Omega_X(X)$ only. 
\end{remark}

\begin{example}
If~$\bfT=\bfG$ is the category of groups, and~$G$ is any group, then the~$G$--module~$\Omega_G(G)$ corresponds to the~$\bbZ G$--module~$\rmI G$, the augmentation ideal of the group ring. For a group homomorphism~$\phi\colon U\to G$ over~$G$, the base change~$\Ab(\phi_*)$ from~$U$--modules to~$G$--modules is given by~\hbox{$M\mapsto\bbZ G\otimes_{\bbZ U}M$}, so that the~$\bbZ G$--module~\hbox{$\bbZ G\otimes_{\bbZ U}\rmI U$} corresponds to the~$G$--module~$\Omega_G(U)$. In particular, we recover the isomorphisms
\[
G^\ab\cong\Omega(G)\cong\bbZ\otimes_{\bbZ G}\rmI G.
\]
Compare with Quillen's notes~\cite[II.5]{Quillen:HA} or Frankland's exposition~\cite[5.1]{Frankland}, for instance.
\end{example}

We will now turn our attention to Beck modules for the theories of racks and quandles, where it is, in general, no longer possible to describe these objects as modules over a ring.


\section{Rack and quandle modules}\label{sec:beck_quan}

Beck modules in the categories of racks and quandles have been studied by Jackson~\cite{Jackson}. The following result rephrases Theorem~2.2 in{\it~loc.cit.}.

\begin{proposition}\label{prop:Jackson}
If~$X$ is a rack, a rack module~$M$ over~$X$ is a family~$(\,M(x)\,|\,x\in X\,)$ of abelian groups together with homomorphisms
\[
M(x)
\overset{\epsilon(x,y)}{\xrightarrow{\hspace*{2em}}}
M(x\rhd y)
\overset{\alpha(x,y)}{\xleftarrow{\hspace*{2em}}}
M(y)
\]
for each pair~$(x,y)$ of elements, such that the three conditions~{\upshape (M1)}, {\upshape(M2)}, and~{\upshape(M3)} below are satisfied.
\end{proposition}

\begin{itemize}[noitemsep,leftmargin=4em]
\item[(M1)] The homomorphisms~$\alpha(x,y)$ are invertible and satisfy
\[
\alpha(x,y\rhd z)\alpha(y,z)=\alpha(x\rhd y,x\rhd z)\alpha(x,z)
\]
for all~$x$,~$y$, and~$z$.

\item[(M2)] The~$\alpha$'s and~$\epsilon$'s commute whenever it makes sense.

\item[(M3)] We have~$\epsilon^2=(\id-\alpha)\epsilon$ whenever it makes sense.
\end{itemize}

\begin{remark}
It follows from the condition~(M1) that a rack module~$M$ comes with~(non-canonical) isomorphisms~$M(y)\cong M(z)$ whenever the elements~$y$ and~$z$ are in the same orbit.
\end{remark}

\begin{remark}
One might wonder if the situation can be described more concisely using the compositions~$\alpha^{-1}\epsilon$.
\end{remark}

The following result rephrases~\cite[Thm.~2.6]{Jackson}.

\begin{proposition}
If~$X$ is a quandle, a quandle module~$M$ over~$X$ is a rack module~$M$ such that, in addition to~{\upshape(M1)},~{\upshape(M2)}, and~{\upshape(M3)}, also the following condition~{\upshape(M4)} is satisfied.
\end{proposition}

\begin{itemize}[noitemsep,leftmargin=4em]
\item[(M4)] We have
\[
\epsilon(x,x)=\id_{M(x)}-\alpha(x,x)
\]
as endomorphisms of~$M(x)$ for all~$x$ in~$X$.
\end{itemize}

\begin{examples}
Given any abelian rack or quandle~$A$ in the sense of Section~\ref{sec:abel_quan}, thought of as a module over the terminal object, its pullback~$X\times A$ is a~`constant' Beck module over~$X$. These modules are `trivial' if, in addition, we have~$\alpha=\id$ and~$\epsilon=0$. More generally, the condition~$\alpha=\id$ only forces~$\epsilon^2=0$, and these modules may be called~`differential.' At the opposite extreme, if we have~$\epsilon=0$, then we are left with the automorphism~$\alpha$, and an `automorphic' module.
\end{examples}

\begin{remark}\label{rem:extension}
If~$M$ is a rack module over a rack~$X$, then the disjoint union of the family~$(\,M(x)\,|\,x\in X\,)$ can be turned into a rack such that the canonical map to~$X$ is a rack morphism. A formula for the operation is given by
\begin{equation}\label{eq:formula}
m\rhd n=\epsilon(x,y)m+\alpha(x,y)n,
\end{equation}
when~$m\in M(x)$ and~$n\in M(y)$, generalizing~\eqref{eq:sum}. If~$X$ is a quandle, and~$M$ is a quandle module over it, then this construction will give a quandle. We will denote the resulting object over~$X$ by~$M$ again. This is analogous to Example~\ref{ex:Gmod}, and it provides one equivalence in Proposition~\ref{prop:Jackson}.
\end{remark}

\begin{remark}
If~$X$ is a rack (or a quandle), and~$N$ is an~$X$--module (in the appropriate sense), then an~$X$--derivation from~$X$ into~$N$ is just a section~$\nu\colon X\to N$ of the morphisms~\hbox{$N\to X$}. In other words, it picks out a family~$(\,\nu(x)\in N(x)\,|\,x\in X\,)$ of elements such that
\[
\nu(x\rhd y)=\epsilon(x,y)\nu(x)+\alpha(x,y)\nu(y)
\]
holds. The right hand side equals~$\nu(x)\rhd\nu(y)$ by~\eqref{eq:formula}. Of course, if~$N$ is constant, this formula simplifies to~$\nu(x\rhd y)=\epsilon\nu(x)+\alpha\nu(y)$. More generally we have~$X$--derivations from~$Y$ into~$N$, where~$Y$ is any rack (or quandle) over~$X$.
\end{remark}


\section{Alexander--Beck modules of knots}\label{sec:beck_knot}

We are now ready to apply the general theory of quandle modules to the fundamental quandles of knots. 

\begin{definition}\label{def:AB_K}
Let~$K$ be a knot with knot quandle~$\rmQ K$. The {\em Alexander--Beck module} of~$K$ is the~$\rmQ K$--module~$\Omega_{\rmQ K}(\rmQ K)$.
\end{definition}

We can find a presentation of the Alexander--Beck module of a knot from any diagram of the knot.

\begin{proposition}
Let~$K$ be a knot. For any diagram, the Alexander--Beck $\rmQ K$--module~$\Omega_{\rmQ K}(\rmQ K)$ is a cokernel of the homomorphism
\[
\Omega_{\rmQ K}(\rmF\bfQ\{\text{\upshape arcs}\})\longleftarrow\Omega_{\rmQ K}(\rmF\bfQ\{\text{\upshape crossings}\})
\]
between free~$\rmQ K$--modules that sends a crossing as in Figure~\ref{fig:crossing} to the element
\begin{equation}\label{eq:relation}
\epsilon(x,y)x+\alpha(x,y)y-z
\end{equation}
in~$\Omega_{\rmQ K}(\rmF\bfQ\{\text{\upshape arcs}\})(z)$.
\end{proposition}

Using the rack structure on~$\rmQ K$--modules from Remark~\ref{rem:extension}, the relation~\eqref{eq:relation} can be written~$x\rhd y=z$, of course.

\begin{proof}
Recall that the knot quandle~$\rmQ K$ has a presentation as a coequalizer
\[
\xymatrix@1{
\rmQ K
&\rmF\bfQ\{\text{arcs}\}\ar[l]
&\rmF\bfQ\{\text{crossings}\}.\ar@<+.4ex>[l]\ar@<-.4ex>[l]
}
\]
The (relative) abelianization functor~$\Omega_{\rmQ K}$ is a left-adjoint, and it therefore preserves colimits such as coequalizers. 
\end{proof}


\begin{proposition}\label{prop:determines}
The Alexander--Beck module of a knot determines its classical Alexander module.
\end{proposition}

\begin{proof}
According to~\eqref{eq:determines}, the absolute abelianization~$\Omega(Q)$ of a quandle~$Q$ is determined by the relative abelianization~$\Omega_Q(Q)$ as the pushforward along the unique morphism~$Q\to\star$ to the terminal object. For~$Q=\rmQ K$ t®his means that the Alexander--Beck module~$\Omega_{\rmQ K}(\rmQ K)$ determines the~$\bbZ[\rmA^{\pm1}]$--module~$\Omega(\rmQ K)$. According to Proposition~\ref{prop:classical}, the latter is isomorphic to the extended Alexander module of the knot, and the classical Alexander module is canonically a retract of it.
\end{proof}

In Example~\ref{ex:Conway}, we will see that the converse to the statement in Proposition~\ref{prop:determines} does not hold: the Alexander--Beck module of a knot is a better invariant than its classical Alexander module. This statement is justified by our main result:


\begin{theorem}\label{thm:main}
A knot is trivial if and only its Alexander--Beck module is a free module over its knot quandle.
\end{theorem}

\begin{proof}
One direction is easy: if~$K=U$ is the unknot, then~$\rmQ U=\rmF\bfQ_1$ is a free quandle on one generator. This is the singleton with the unique quandle structure, which is also the terminal object in the category~$\bfQ$ of quandles. It follows that there is no difference between the absolute and the relative abelianization, and we get
\[
\Omega_{\rmQ U}(\rmQ U)=\Omega(\rmQ U)=\Omega(\rmF\bfQ_1),
\]
and this corresponds to the free~$\bbZ\bfQ$--module of rank~$1$ as we have seen in Section~\ref{sec:Beck}.

As for the other direction: let~$K$ be a knot such that the~$\rmQ K$--module $\Omega_{\rmQ U}(\rmQ U)$ is free. The left adjoint~$\Phi_*$ in the adjunction
\begin{equation}\label{eq:adj1}
\xymatrix@1{
\Phi_*\colon\bfQ\ar@<+.4ex>[r]
&
\bfG\colon\Phi^*\ar@<+.4ex>[l]
}
\end{equation}
between the category~$\bfQ$ of quandles and the category~$\bfG$ of groups sends the knot quandle~$\rmQ K$ to the knot group~$\pi K=\Phi_*\rmQ K$. This adjunction induces an adjunction
\[
\xymatrix@1{
\Phi_*\colon\bfQ_{\rmQ K}\ar@<+.4ex>[r]
&
\bfG_{\pi K}\colon\Psi^*\ar@<+.4ex>[l]
}
\]
between the slice categories. See~\cite[Prop.~4.1]{Frankland}. The left adjoint~$\Phi_*$ sends an arrow~\hbox{$P\to\rmQ K$} to the arrow~$\Phi_*P\to\Phi_*\rmQ K=\pi K$ induced by~$\Phi_*$, justifying the notation. The right adjoint~$\Psi^*$ is the composition of~$\Phi^*$ with the pullback functor along the unit
\[
u\colon \rmQ K\longrightarrow\Phi^*\Phi_*\rmQ K=\Phi^*\pi K
\]
of the adjunction~\eqref{eq:adj1} at~$\rmQ K$,
\[
\xymatrix{
\Psi^*G\ar[r]\ar[d]&\Phi^*G\ar[d]\\
\rmQ K\ar[r]&\Phi^*\pi K.
}
\]
This right adjoint~$\Psi^*$ restricts to a functor between abelian group objects, and that functor has a left adjoint~$\Psi_*$,
\begin{equation}\label{eq:adj3}
\xymatrix@1{
\Psi_*\colon\Ab(\bfQ_{\rmQ K})\ar@<+.4ex>[r]
&
\Ab(\bfG_{\pi K})\colon\Psi^*\ar@<+.4ex>[l]
}.
\end{equation}
See~\cite[Prop.~4.2]{Frankland}. We get a commutative diagram
\[
\xymatrix@C=4em{
\bfQ_{\rmQ K} & \bfG_{\pi K}\ar[l]_-{\Psi^*}\\
\Ab(\bfQ_{\rmQ K})\ar[u]^R & \Ab(\bfG_{\pi K})\ar[l]^-{\Psi^*}\ar[u]_R}
\]
of right adjoints, where~$R$ denotes the forgetful functor, so that the diagram
\[
\xymatrix@C=4em{
\bfQ_{\rmQ K}\ar[r]^-{\Phi_*}\ar[d]_{\Omega_{\rmQ K}} & \bfG_{\pi K}\ar[d]^{\Omega_{\pi K}}\\
\Ab(\bfQ_{\rmQ K})\ar[r]_-{\Psi_*} & \Ab(\bfG_{\pi K})
}
\]
of left adjoints also commutes. Evaluated at the identity~$\rmQ K=\rmQ K$, this proves
\[
\Psi_*\Omega_{\rmQ K}(\rmQ K)=\Omega_{\pi K}(\pi K).
\]
If the left adjoint~$\Psi_*$ preserves free objects, then our assumption implies that the~$\pi K$--module~$\Omega_{\pi K}(\pi K)$ is free. Let us assume for a moment that this is the case, to see that this allows us to finish the proof as follows: under the equivalence between the category of~$\pi K$--modules and the category of~$\bbZ\pi K$--modules, this module corresponds to the~$\bbZ\pi K$--module~$\rmI\pi K$, the augmentation ideal in the group ring~$\bbZ\pi K$. It follows that~$\rmI\pi K$ is a free~$\bbZ\pi K$--module. Then the defining exact sequence
\[
0\longleftarrow\bbZ\longleftarrow\bbZ\pi K\longleftarrow\rmI\pi K\longleftarrow0
\]
is a free resolution of the trivial~$\bbZ\pi K$--module~$\bbZ$. It follows that the knot group~$\pi K$ has cohomological dimension~$1$. By the characterization of groups of cohomological dimension~$1$, due to Stallings~\cite{Stallings} and Swan~\cite{Swan}, the knot group~$\pi K$ is free, and therefore necessarily cyclic. By~\cite{Papakyriakopoulos} again, the knot~$K$ is trivial, as desired.

It remains to be seen that the left adjoint~$\Psi_*$ preserves free objects. To do so, let us choose a free module~$R$ over~$Q=\rmQ K$ with basis a set~$S\to Q$ over~$Q$, so that there is an adjunction bijection
\begin{equation}\label{eq:adj2}
\Hom_{\Ab(\bfQ_Q)}(R,M)\cong\bfS_Q(S,M)
\end{equation}
for all~$Q$--modules~$M$. We need to show that the image~$\Psi_*R$ is a free module over the group~$\pi=\pi K$. So we compute
\[
\Hom_{\Ab(\bfG_\pi)}(\Psi_*(R),N)\cong\Hom_{\Ab(\bfQ_Q)}(R,\Psi^*(N))\cong\bfS_Q(S,\Psi^*(N)),
\]
the first by the adjunction~\eqref{eq:adj3}, and the second by the adjunction~\eqref{eq:adj2}. We are done if the right hand side is also naturally isomorphic to~$\bfS_\pi(L(S),N)$ for some functor~$L$ that takes sets over~$Q$ to sets over~$\pi$. In other words, we need to know that the composition
\[
\Ab(\bfG_\pi)\overset{\Psi^*}{\longrightarrow}\Ab(\bfQ_Q)\overset{R}\longrightarrow\bfS_Q
\]
has a left adjoint, and that is clear: it is a composition of right adjoints, so that the composition of their left adjoints is the left adjoint.
\end{proof}

\begin{remark}
Theorem~\ref{thm:main} may remind some readers of Eisermann's characterization~\cite{Eisermann} of the unknot. However, the similarities are limited to the extent that both results give invariants that detect the unknot, and that both of these invariants are constructed starting from quandles. The arguments are entirely different: we do the algebra, whereas for Eisermann the passage from closed knots to open knots is essential. Also, his result is homological in nature, and he uses higher quandle cohomology with values in abelian groups. This present paper does not involve homology, and our unknot-detector takes values in the somewhat more elaborate Beck modules for quandles. This structure is lost when passing from a quandle to its cohomology groups. It appears rather more likely that the present result, when embedded into a suitable derived context~\cite{Szymik:Q=Q}, implies Eisermann's, but so far this has no been confirmed.
\end{remark}

\begin{example}\label{ex:Conway}
The Alexander polynomial is trivial for many knots. Examples include
the Seifert knot~\cite{Seifert}, 
all untwisted Whitehead doubles~\cite{Whitehead} of knots,
the Kinoshita--Terasaka knot~\cite{Kinoshita+Terasaka}, and
the Conway knot~\cite{Conway}.~(See also~\cite{Garoufalidis+Teichner}.) 
On the other hand, these knots are all non-trivial, and their Alexander--Beck modules are not free~(Theorem~\ref{thm:main}). 
The Alexander--Beck module is therefore a stronger invariant than the plain Alexander module.
\end{example}


\section*{Acknowledgments}

This research owes a lot to a stay at the Hausdorff Institute for Mathematics.~I thank Scott Carter, Martin Frankland, Mike Hill, Haynes Miller, Mark Powell, and Peter Teichner for stimulating discussions on that and other occasions, and the referee for their careful reading and comments.



\vfill

\parbox{\linewidth}{%
Department of Mathematical Sciences\\
NTNU Norwegian University of Science and Technology\\
7491 Trondheim\\
NORWAY\\
\phantom{ }\\
\href{mailto:markus.szymik@ntnu.no}{markus.szymik@ntnu.no}\\
\href{https://folk.ntnu.no/markussz}{folk.ntnu.no/markussz}
}


\end{document}